\documentclass[reqno]{amsart}
\usepackage{amscd}
\usepackage{amssymb}
\usepackage[dvips]{graphics}
\usepackage[margin=1in,includeheadfoot]{geometry}
\usepackage{url}
\usepackage{hyperref}
\usepackage{cite}
\usepackage{color}

\def\beq{\begin{equation}}
\def\eeq{\end{equation}}
\def\ba{\begin{array}}
\def\ea{\end{array}}

\def\R{\mathbb R}

\newtheorem{thm}{Theorem}[section]
\newtheorem{lm}[thm]{Lemma}

\newtheorem{crl}[thm]{Corollary}

\theoremstyle{definition}
\newtheorem{rem}[thm]{Remark}

\theoremstyle{remark}

\begin{document}
\pagestyle{plain}\today
\title{The eternal solutions of parabolic equations with boundary condition}
\author{Jingqi Liang}
\email{liangjq2017@sjtu.edu.cn}
\address{Jingqi Liang: Institute of Natural Sciences, Shanghai Jiao Tong University, Shanghai 200240, China}

%\author{\small{Jingqi Liang}\\
% \small{Institute of Natural Sciences, Shanghai Jiao Tong University}\\
 %\small{Shanghai, China}}

 \author{Lidan Wang}
\email{wanglidan@ujs.edu.cn}
\address{Lidan Wang: School of Mathematical Sciences, Jiangsu University, Zhenjiang 221013, China}

\thanks{
The first author is supported by Natural Science Foundation of Shanghai, No.24ZR1440700. The corresponding author is Lidan Wang and the second author is supported by NSFC of China, No.12401135.}

\begin{abstract}
In this paper, we study the parabolic equations of the form
$$
\left\{
\begin{array}{rcll}
Lu(y,t) &=& f, \qquad &(y,t)\in Q,\\
u(y,t)&=& 0,  \qquad &(y,t)\in \partial Q, \\
u(y,t)&& \hspace{-8mm}\mbox{is uniformly bounded from below}, \qquad &(y,t)\in Q,
\end{array}
\right.
$$
where $Q=\Omega\times\mathbb{R}\subset\mathbb{R}^{n+1}$ and $\Omega\subset\mathbb{R}^{n}$ is a bounded Lipschitz domain with $0\in\Omega$. Here $L$ is a general second order uniformly parabolic differential operator in non-divergence form or divergence form. For $f=0$, we establish the structure of the solution space, which is one dimensional and the solutions in this space grow exponentially at one end and decay exponentially at the other. For $f\neq0$, we show that all solutions can be presented by the solutions corresponding to the homogenous equations($f=0$) and a bounded special solution of the inhomogeneous equations. Our method is based on maximum principle in $Q$ and the Harnack type inequalities.

\end{abstract}

\maketitle


\
\


\section{Introduction}
The solutions of elliptic equations on unbounded domains with boundary condition have been studied extensively, see for examples \cite{B,CC,GP,LN,M}. It is well known that any positive solution of the equation $\Delta u=0$ on $\R\times(0,\pi)$ with zero boundary condition can be presented by two
linearly independent positive harmonic functions $e^{x}\sin y$ and $e^{-x}\sin y$. This result was extended to second order elliptic operators by Bao, Wang and Zhou \cite{BWZ}. More precisely, Bao et al. considered the second order homogeneous equations on unbounded cylinders with zero boundary condition and proved that all positive solutions are linear combinations of two special positive solutions with exponential growth at one end and exponential decay at the other. After that, Wang, Wang and Zhou \cite{LWZ} generalized the results of Bao et al. \cite{BWZ} to second order elliptic equations with lower order terms. Moreover, they established that the solutions of the inhomogeneous equations are generated by the solutions of the corresponding homogenous equations and a bounded special solution of the inhomogeneous equations. Later, Wang, Wang and Zhou \cite{LWZ1} studied the fully nonlinear inhomogeneous elliptic equations on unbounded cylinders with zero boundary condition and showed that two special solution spaces (exponential growth at one
end and exponential decay at the another) are one dimensional, independently. While the solutions in the third solution space can be controlled by the solutions
in the other two special solution spaces under some conditions, respectively. Hang and Lin \cite{HL1} proved that, for a class of divergence form elliptic equations on unbounded cylinders with zero boundary, the space of fixed order exponential growth solutions is of finite dimension. While the authors in \cite{LWZ3} proved that, for a class of non-divergence form elliptic equations on unbounded cylinders with zero boundary, the space of fixed order exponential growth solutions is  finite dimensional. For more related works about the elliptic equations with zero boundary condition, we refer the readers to \cite{HL,H3,L1,LWZ2}.

A natural generalization is to consider the solutions of parabolic equations on unbounded domains with boundary condition.  For example, Feng \cite{F} established the parabolic analogs of Hang-Lin's results \cite{HL1}. Namely, Feng proved that, for a class of parabolic equations with zero boundary condition, the dimension of the solutions with exponential growth is finite. Corresponding to the Laplace equation $\Delta u=0$ on $\R\times(0,\pi)$ with zero boundary condition, one gets easily that the heat equation $u_t-\Delta u=0$ on $\R\times(0,\pi)$ with zero boundary condition has a positive solution  $e^{-t}\sin x$, and any positive solution can be presented by this positive solution. To the best of our knowledge, there are no such results for the inhomogeneous parabolic equations with lower order terms. Motivated by the works mentioned above, similar to the case of elliptic equations, in this paper, we would like to study a class of inhomogeneous parabolic equations with zero boundary condition and discuss the structure of solutions. More precisely, we study the following parabolic equations:
\begin{equation}\label{pe1}
\left\{
\begin{array}{rcll}
Lu(y,t) &=& f, \qquad &(y,t)\in Q,\\
u(y,t)&=& 0,  \qquad &(y,t)\in \partial Q, \\
u(y,t)&& \hspace{-8mm}\mbox{is uniformly bounded from below}, \qquad &(y,t)\in Q,
\end{array}
\right.
\end{equation}
where $Q=\Omega\times(-\infty,+\infty)=\{x=(y,t)\in\mathbb{R}^{n+1}|y=(y_{1},y_{2},\cdots,y_{n})\in\Omega,~t\in(-\infty,+\infty)\}$, $\Omega\subset\mathbb{R}^{n}$ is a bounded Lipschitz domain and $0\in\Omega$. Here $L$ is a second order uniformly parabolic differential operator in nondivergence form or divergence form, i.e.
$$
Lu(y,t)=\frac{\partial u}{\partial t}(y,t)-\sum_{i,j=1}^{n}a_{ij}(y,t)\frac{\partial^{2}u(y,t)}{\partial y_{i}\partial y_{j}}+\sum_{i=1}^{n}b_{i}(y,t)\frac{\partial u(y,t)}{\partial y_{i}}+c(y,t)u(y,t),
$$
or
$$
Lu(y,t)=\frac{\partial u}{\partial t}(y,t)-\sum_{i,j=1}^{n}\frac{\partial}{\partial y_{i}}(a_{ij}(y,t)\frac{\partial u(y,t)}{\partial y_{j}})+\sum_{i=1}^{n}b_{i}(y,t)\frac{\partial u(y,t)}{\partial y_{i}}+c(y,t)u(y,t).
$$

Throughout the paper, we only prove the results in the nondivergence form. In this case, we consider the strong solution $u\in W_{n+1,\text{loc}}^{2,1}(Q)\cap C(\bar{Q})$. We always assume that $a_{ij}\in C({\bar{Q}})$ with $a_{ij}(y,t)=a_{ji}(y,t)$ satisfies the uniformly parabolic condition: there exist $\lambda,\Lambda>0$ such that
\begin{eqnarray*}
\lambda|\xi|^{2}\leq a_{ij}(y,t)\xi_{i}\xi_{j}\leq\Lambda|\xi|^{2},\quad (y,t)\in Q,~\xi\in \R^{n},
\end{eqnarray*}
 and $b_{i},~c,~f$ satisfy
\begin{eqnarray*}
b_{i}\in L^{\infty}(Q),\quad\|b_{i}\|_{L^{\infty}(Q)}\leq\Lambda,
\end{eqnarray*}
\begin{eqnarray*}
c\in L^{\infty}(Q),\quad\|c\|_{L^{\infty}(Q)}\leq\Lambda,\quad c\geq0,\quad(y,t)\in Q.
\end{eqnarray*}
\begin{eqnarray*}
f\in L^{n+1}_{\text{loc}}(Q), \quad \|f\|_{L_{*}^{n+1}(Q)}:=\sup_{t\in\mathbb{R}}\|f\|_{L^{n+1}(Q_{(t,t+2)})}<+\infty.
\end{eqnarray*}
If $f=0$, then the problem (\ref{pe1}) turns into the problem
\begin{eqnarray*}
\left\{
\begin{array}{rcll}
Lu(y,t) &=& 0, \qquad &(y,t)\in Q,\\
u(y,t)&=& 0,  \qquad &(y,t)\in \partial Q, \\
u(y,t)&>& 0, \qquad &(y,t)\in Q,
\end{array}
\right.
\end{eqnarray*}
where we have used the maximum principle in $Q$, see Lemma \ref{mpqf} below. We denote by $\tilde{U}$ the solution set of the problem $(\ref{pe1})$. In particular, if $f=0$, we use $U$ to denote the solution set of problem $(\ref{pe1})$, which means that $\tilde{U}=U$ for $f=0$.

To state our results, we first give some notations. For $x_{1}=(y_{1},t_{1}),~x_{2}=(y_{2},t_{2})\in\mathbb{R}^{n+1}$, we define the parabolic distance:
$\text{dist}_{p}(x_{1},x_{2})=\max\{|y_{1}-y_{2}|,|t_{1}-t_{2}|^{\frac{1}{2}}\}.
$
Then the parabolic neighborhood can be defined by
$O_{p}(x_{0},\delta)=\{x\in\mathbb{R}^{n+1}|\text{dist}_{p}(x,x_{0})<\delta\}.
$
Moreover, for any $E\subset\mathbb{R}$, $Q_{E}:=\Omega\times E$, $\partial_{l}Q_{E}:=\partial\Omega\times E=\{(y,t)|y\in\partial\Omega,~t\in E\}$. If $E=(a,b)\subset\mathbb{R}$ for $-\infty<a<b\leq+\infty$, we denote $Q_{(a,b)}:=\Omega\times (a,b)$,  $\partial_{l}Q_{(a,b)}:=\partial\Omega\times(a,b)$, $\partial_{b}Q_{(a,b)}:=\Omega\times\{t=a\}$, $\partial_{c}Q_{(a,b)}:=\partial\Omega\times\{t=a\}$, $\partial_{p}Q_{(a,b)}:=\partial_{l}Q_{(a,b)}\cup\partial_{b}Q_{(a,b)}\cup\partial_{c}Q_{(a,b)}$. For $Q=\Omega\times(-\infty,+\infty)$, $\partial Q=\partial_{l}Q=\partial\Omega\times(-\infty,+\infty)$. For any $t\in\mathbb{R}$, let $Q_{t}:=Q_{\{t\}}$, $Q^{+}_{t}:=Q_{(t,+\infty)}$, $Q^{-}_{t}:=Q_{(-\infty,t)}$, $Q^{+}:=Q_{0}^{+}$, $Q^{-}=Q_{0}^{-}$.

 For $u\in\tilde{U}$, we write $\hat{u}(t):=\sup\limits_{y\in\Omega}u^{+}(y,t)$, $t\in\mathbb{R}$, where $u^{+}=\max\{u,0\}$ and $m(u):=\inf\limits_{t\in\mathbb{R}}\hat{u}(t)$. Clearly, for $u\in U$, $\hat{u}(t)=\sup\limits_{y\in\Omega}u(y,t)$, $t\in\mathbb{R}$.

Now we state our main results. The first one is about the structure of the positive solution set $U$.
\begin{thm}\label{se0}
For the problem (\ref{pe1}) with $f=0$, the positive solution set $U$ is well defined.  Moreover, we have $U=\{u|~u=av,v\in U, a>0\}$.
\end{thm}

The following two theorems are about the asymptotic behavior of positive solutions at infinity.
\begin{thm}\label{th2}
For the problem (\ref{pe1}) with $f=0$, there exist constants $\alpha,~\beta,~C,~C'$ depending only on $n,\lambda,\Lambda,\Omega$ such that, for any $u\in U$,
\begin{equation}\label{e1}
C\hat{u}(0)e^{\beta|t|}\leq\hat{u}(t)\leq C'\hat{u}(0)e^{\alpha|t|},\quad t\in(-\infty,0),
\end{equation}
\begin{equation}\label{e2}
\frac{1}{C'}\hat{u}(0)e^{-\alpha t}\leq\hat{u}(t)\leq \frac{1}{C}\hat{u}(0)e^{-\beta t},\quad t\in(0,+\infty).
\end{equation}
\end{thm}

\begin{thm}\label{th3}
Let $u$ be a solution of the following problem
\begin{equation*}
\left\{
\begin{array}{rcll}
Lu(y,t) &=& 0, \qquad &(y,t)\in Q^{+},\\
u(y,t)&=& 0,  \qquad &(y,t)\in \partial_{l} Q^{+}, \\
u(y,t)&>& 0, \qquad &(y,t)\in Q^{+}.
\end{array}
\right.
\end{equation*}
Then for any $w\in U$, there exist constants $\alpha>0$ depending only on $n,\lambda,\Lambda,\Omega$, and $K,~C>0$ depending only on $u,n,\lambda,\Lambda,\Omega$ such that
$$|u(y,t)-Kw(y,t)|\leq Ce^{-\alpha t}w(y,t),\quad (y,t)\in Q_{(1,+\infty)}.
$$
\end{thm}

Finally, we establish the structure of the solution set  $\tilde{U}$.
\begin{thm}\label{th4}
For the problem $(\ref{pe1})$ with $f\neq 0$, the set of solutions bounded from below $\tilde{U}$ can be represented by, for any $u\in U$,
$$\tilde{U}=U^{0}+U=\{u_{0}+au|a\geq0\},
$$
where $U^{0}=\{u_{0}\}$ is the bounded solution of $Lu_{0}=f$ in $Q$ with zero boundary condition.
\end{thm}

This paper is organized as follows. In Section 2, we mainly prove a maximum principle in $Q$. In Section 3, we establish a Harnack inequality and a comparison theorem in our form. In Section 4, we show the structure and asymptotic behavior of positive solutions(Theorem \ref{se0}-Theorem \ref{th3}). In Section 5, we demonstrate the structure of solutions bounded from below(Theorem \ref{th4}).

\section{Maximum principle}
In this section, we are devoted to proving the maximum principle in $Q$. First, we introduce a decay lemma, which plays a key role in our proof.
\begin{lm}\label{mp02f}
There are constants $0<\delta<1$ and $0<\varepsilon_{0}<1$ such that if $u$ satisfies
\begin{eqnarray*}
\left\{
\begin{array}{rcll}
Lu(x) &\leq& f(x), \qquad &x\in Q_{(0,2)},\\
u(x)&\leq& 0,  \qquad &x\in\partial_{l}Q_{(0,2)}, \\
u(x)&\leq& 1, \qquad &x\in\partial_{p}Q_{(0,2)}\backslash\partial_{l}Q_{(0,2)},
\end{array}
\right.
\end{eqnarray*}
where $\|f\|_{L^{n+1}(Q_{(0,2)})}\leq\varepsilon_{0}$, then
$$u(y,1)\leq1-\delta, \quad y\in\Omega,
$$
where $\delta$ depends only on $n,\lambda,\Lambda$ and $\Omega$.
\end{lm}
\begin{proof}
Let $w(x)$ be a solution of
\begin{eqnarray*}
\left\{
\begin{array}{rcll}
Lw(x)&=& f(x), \qquad &x\in Q_{(0,2)}, \\
 w(x)&=& \mbox{max}\{u(x), 0\},\qquad &x\in\partial_{p}Q_{(0,2)}.
\end{array}
\right.
\end{eqnarray*}
By maximum principle in \cite[Theorem 5.1]{Li}, we have
$$u(x)\leq w(x)\leq1+C\|f\|_{L^{n+1}(Q(0,2))}\leq 1+C\varepsilon_{0},\quad x\in Q_{(0,2)},
$$
where $C$ depends only on $n,\lambda,\Lambda,\Omega$. By the boundary H\"{o}lder estimate in \cite[Theorem 5.1]{Gr}, there exist constants $C_{0}>0$ and $0<\alpha<1$ depending only on $n,\lambda,\Lambda,\Omega$ such that
$$[w]_{C^{\alpha}(Q_{(\frac{1}{2},\frac{3}{2})})}\leq C_{0}.
$$
Since $w(x)=0$ on $\partial_{l}Q_{(0,2)}$, it follows that for any $x=(y,t)\in Q_{(\frac{1}{2},\frac{3}{2})}$,
$$|w(x)|\leq C_{0}\text{dist}^{\alpha}(x,\partial_{l}Q_{(\frac{1}{2},\frac{3}{2})})
=C_{0}\text{dist}^{\alpha}(y,\partial\Omega).
$$
We take $\sigma_{0}$ small enough such that $C_{0}\sigma_{0}^{\alpha}\leq\frac{1}{2}$, then for any $x=(y,t)\in Q_{(\frac{1}{2},\frac{3}{2})}$ with $\text{dist}(y,\partial\Omega)\leq \sigma_{0}$, we have
$$|w(x)|\leq C_{0}\sigma_{0}^{\alpha}\leq\frac{1}{2}, \quad x\in Q_{(\frac{1}{2},\frac{3}{2})}.
$$
Let $Q'_{(\frac{1}{2},\frac{5}{4})}=\{x=(y,t)\in Q_{(\frac{1}{2},\frac{5}{4})}|\text{dist}(y,\partial\Omega)> \frac{1}{2} \sigma_{0}\}$, $Q'_{(\frac{4}{3},\frac{3}{2})}=\{x=(y,t)\in Q_{(\frac{4}{3},\frac{3}{2})}|\text{dist}(y,\partial\Omega)> \frac{1}{2}\sigma_{0}\}$ and $Q''_{(\frac{4}{3},\frac{3}{2})}=\{x=(y,t)\in Q_{(\frac{4}{3},\frac{3}{2})}|\sigma_{0}>\text{dist}(y,\partial\Omega)> \frac{1}{2}\sigma_{0}\}$. Clearly we have that $1+C\varepsilon_{0}-w$ is nonnegative and satisfies $L(1+C\varepsilon_{0}-w)=-f(x)+c(x)(1+C\varepsilon_{0})\geq -f(x)$ in $Q_{(0,2)}$. This means that $1+C\varepsilon_{0}-w$ is a nonnegative supersolution of $Lu=-f$ in $Q_{(0,2)}$. Moreover, $\frac{1}{2}+C\varepsilon_{0}\leq 1+C\varepsilon_{0}-w\leq\frac{3}{2}+C\varepsilon_{0}$ in the set $\{x=(y,t)\in Q_{(\frac{1}{2},\frac{3}{2})}|\text{dist}(y,\partial\Omega)\leq \sigma_{0}\}$. Then we apply the weak Harnack inequality in \cite[Theorem 3.1]{Gr} to $1+C\varepsilon_{0}-w$ in $Q_{(\frac{1}{2},\frac{3}{2})}$, and we obtain that for some $p>0$,
\begin{eqnarray*}
\frac{1}{2}C_{1}\sigma_{0}^{n}\leq C_{1}\sigma_{0}^{n}\left(\frac{1}{2}+C\varepsilon_{0}\right)&\leq &\left\{\frac 1{|Q'_{(\frac{4}{3},\frac{3}{2})}|}\int_{Q''_{(\frac{4}{3},\frac{3}{2})}}
(1+C\varepsilon_{0}-w)^{p}dx\right\}^{\frac{1}{p}}\\
&\leq &\left\{\frac 1{|Q'_{(\frac{4}{3},\frac{3}{2})}|}\int_{Q'_{(\frac{4}{3},\frac{3}{2})}}
(1+C\varepsilon_{0}-w)^{p}dx\right\}^{\frac{1}{p}}\\&\leq & C\left\{\inf\limits_{Q'_{(\frac{1}{2},\frac{5}{4})}}(1+C\varepsilon_{0}-w)+\|f\|_{L^{n+1}(Q_{(\frac{1}{2},\frac{3}{2})})}\right\}\\
&\leq& C\left\{\inf\limits_{Q'_{(\frac{1}{2},\frac{5}{4})}}(1+C\varepsilon_{0}-w)+\varepsilon_{0})\right\},
\end{eqnarray*}
where $C_{1}$ is a constant depending only on $n$ and $\Omega$. Therefore, by taking $\varepsilon_{0}\leq\frac{C_{1}\sigma_{0}^{n}}{4C(1+C)}$, for $x=(y,t)\in Q_{(\frac{1}{2},\frac{5}{4})}$ with $\text{dist}(y,\partial\Omega)> \frac{1}{2} \sigma_{0}$, we have $1-w(y,t)\geq \frac{C_{1}\sigma_{0}^{n}}{2C}-(C+1)\varepsilon_{0}\geq\frac{C_{1}\sigma_{0}^{n}}{4C}>0$. Noting that for $x=(y,t)\in Q_{(\frac{1}{2},\frac{3}{2})}$ with $\text{dist}(y,\partial\Omega)\leq \sigma_{0}$, $1-w(y,t)\geq \frac{1}{2}$. Let $\delta=\min\{\frac{1}{2},\frac{C_{1}\sigma_{0}^{n}}{4C}\}$, we obtain that
$$u(y,1)\leq w(y,1)\leq 1-\delta, \quad y\in\Omega.
$$
\end{proof}
\begin{rem}\label{rem3.2f}
In fact, in the proof of Lemma $\ref{mp02f}$,  we can get that
$$\hat{u}(1)=\sup_{y\in\Omega}\{u(y,1),0\}\leq 1-\delta.
$$
\end{rem}
\begin{crl}\label{cormp02f}
There exist constant $0<\delta<1$ and $0<\varepsilon_{0}<1$ such that if $u$ satisfies
\begin{eqnarray*}
\left\{
\begin{array}{rcll}
Lu(x) &\leq& f(x), \qquad &x\in Q_{(0,2)},\\
u(x)&\leq& 0,  \qquad &x\in\partial_{l}Q_{(0,2)},
\end{array}
\right.
\end{eqnarray*}
then
$$\hat{u}(1)\leq(1-\delta)\hat{u}(0)+\frac{1-\delta}{\varepsilon_{0}}\|f\|_{L^{n+1}(Q_{(0,2)})},
$$
where $\delta$ depends only on $n,\lambda,\Lambda$ and $\Omega$.
\end{crl}
\begin{proof}We divide the proof into two cases.

\
\

{\bf Case 1: $f=0$ in $Q_{(0,2)}$.}
By the definition of $\hat{u}$, $\hat{u}$ is a nonnegative function, naturally $\hat{u}(0)\geq0$. If $\hat{u}(0)=0$, it follows that $u\leq0$ on $\partial_{b}Q_{(0,2)}\cup\partial_{c}Q_{(0,2)}$, combining with $u\leq0$ on $\partial_{l}Q_{(0,2)}$, we have $u\leq0$ on $\partial_{p}Q_{(0,2)}$. Then by maximum principle, $u\leq0$ in $Q_{(0,2)}$. Hence $\hat{u}(1)\leq0=\hat{u}(0)$.

If $\hat{u}(0)>0$, we consider the function $\tilde{u}(x)=\frac{u(x)}{\hat{u}(0)}$. It is easy to verify that $\tilde{u}(x)$ satisfies
\begin{eqnarray*}
\left\{
\begin{array}{rcll}
L\tilde{u}(x) &\leq& 0, \qquad &x\in Q_{(0,2)},\\
\tilde{u}(x)&\leq& 0,  \qquad &x\in\partial_{l}Q_{(0,2)}, \\
\tilde{u}(x)&\leq& 1, \qquad &x\in\partial_{p}Q_{(0,2)}\backslash\partial_{l}Q_{(0,2)}.
\end{array}
\right.
\end{eqnarray*}
Then we can apply Lemma $\ref{mp02f}$ and Remark $\ref{rem3.2f}$ to $\tilde{u}(x)$, it follows that
$$\hat{u}(1)\leq(1-\delta)\hat{u}(0).
$$

\
\

{\bf Case 2: $f\neq0$ in $Q_{(0,2)}$.}
We consider the function
\begin{eqnarray*}
\tilde{u}(x)=\frac{\varepsilon_{0}u(x)}{\varepsilon_{0}\hat{u}(0)+\|f\|_{L^{n+1}(Q_{(0,2)})}},
~~ x\in Q_{(0,2)}.
\end{eqnarray*}
Obviously, $\tilde{u}(x)$ satisfies
\begin{eqnarray*}
\left\{
\begin{array}{rcll}
L\tilde{u}(x) &\leq& \tilde{f}(x), \qquad &x\in Q_{(0,2)},\\
\tilde{u}(x)&\leq& 0,  \qquad &x\in\partial_{l} Q_{(0,2)}, \\
\tilde{u}(x)&\leq& 1, \qquad &x\in\partial_{p} Q_{(0,2)}\backslash\partial_{l} Q_{(0,2)},
\end{array}
\right.
\end{eqnarray*}
where
\begin{eqnarray*}
\tilde{f}(x)=\frac{\varepsilon_{0}f(x)}{\varepsilon_{0}\hat{u}(0)+\|f\|_{L^{n+1}(Q_{(0,2)})}},
~~ x\in Q,
\end{eqnarray*}
and satisfies $\|\tilde{f}\|_{L^{n+1}(Q_{(0,2)})}\leq\varepsilon_{0}$. For $\tilde{u}(x)$ in $Q_{(0,2)}$, by Lemma $\ref{mp02f}$ and Remark $\ref{rem3.2f}$, we get that there exists a constant $\delta\in(0,1)$ such that $\hat{\tilde{u}}(1)\leq1-\delta$, i.e.
\begin{eqnarray*}
\hat{u}(1)&\leq&\frac{(1-\delta)}{\varepsilon_{0}}\{\varepsilon_{0}\hat{u}(0)
+\|f\|_{L^{n+1}(Q_{(0,2)})}\}\\
&= &(1-\delta)\hat{u}(0)+\frac{(1-\delta)}{\varepsilon_{0}}
\|f\|_{L^{n+1}(Q_{(0,2)})}.
\end{eqnarray*}

Combining the above two cases, we finish the proof.
\end{proof}

In fact, for any $t_{0}\in(-\infty,+\infty)$, since the diameter of $Q_{(t_{0},t_{0}+2)}$ depends only on $n$ and $\text{diam}(\Omega)$, i.e. $\text{diam}(Q_{(t_{0},t_{0}+2)})$ is independent of $t_{0}$, we can give a general version of Lemma $\ref{mp02f}$ and Corollary $\ref{cormp02f}$.
\begin{crl}\label{mpgf}
There exist constant $0<\delta<1$ and $0<\varepsilon_{0}<1$ such that for any $t_{0}\in(-\infty,+\infty)$, if $u$ satisfies
\begin{eqnarray*}
\left\{
\begin{array}{rcll}
Lu(x) &\leq& f(x), \qquad &x\in Q_{(t_{0},t_{0}+2)},\\
u(x)&\leq& 0,  \qquad &x\in\partial_{l}Q_{(t_{0},t_{0}+2)},
%u(x)&\leq& 1, \qquad &x\in\partial_{p}Q_{(k,k+2)}\backslash\partial_{l}Q_{(k,k+2)},
\end{array}
\right.
\end{eqnarray*}
then
%$$u(y,k+1)\leq1-\delta, \quad y\in\Omega.
%$$
%Furthermore,
$$\hat{u}(t_{0}+1)\leq(1-\delta)\hat{u}(t_{0})+\frac{1-\delta}{\varepsilon_{0}}\|f\|_{L^{n+1}(Q_{(t_{0},t_{0}+2)})},
$$
where $\delta$ depends only on $n,\lambda,\Lambda$ and $\Omega$.
\end{crl}
\begin{lm}[Maximum principle in $Q$]\label{mpqf}
Let $u$ satisfy $Lu(x)\leq f(x)$ for $x\in Q$. If $u(x)$ is bounded from above, then we have
$$\sup_{x\in Q}u^{+}(x)\leq\sup_{x\in\partial Q}u^{+}(x)+C\|f\|_{L^{n+1}_{*}(Q)},
$$
where $C$ depends only on $n,\lambda,\Lambda,\Omega$.
\end{lm}
\begin{proof}
Without loss of generality, we assume that $\sup\limits_{x\in\partial Q}u^{+}(x)=0$. If not, we can consider the function $w(x)=u(x)-\sup\limits_{x\in\partial Q}u^{+}(x)$. Therefore we only need to prove
$$u^{+}(x)\leq C\|f\|_{L^{n+1}_{*}(Q)},\quad x\in Q.
$$

We assume there exists $M$ large enough such that $u(x)\leq M$ for any $x\in Q$, and denote $\|f\|_{L^{n+1}_{*}(Q)}$ by $F$. %For any $k\in \mathbb{Z}$, in order to use Lemma $\ref{mp02f}$, we consider the function
%\begin{eqnarray*}
%\tilde{u}(x)=\frac{\varepsilon_{0}u(x)}{\varepsilon_{0}\hat{u}(k-1)+\|f\|_{L^{n+1}(Q_{(k-1,k+1)})}},
%~~ x\in Q.
%\end{eqnarray*}
%Obviously, $\tilde{u}(x)$ satisfies
%\begin{eqnarray*}
%\left\{
%\begin{array}{rcll}
%L\tilde{u}(x) &\leq& \tilde{f}(x), \qquad &x\in Q_{(k-1,k+1)},\\
%\tilde{u}(x)&\leq& 0,  \qquad &x\in\partial_{l} Q_{(k-1,k+1)}, \\
%\tilde{u}(x)&\leq& 1, \qquad &x\in\partial_{p} Q_{(k-1,k+1)}\backslash\partial_{l} Q_{(k-1,k+1)},
%\end{array}
%\right.
%\end{eqnarray*}
%where
%\begin{eqnarray*}
%\tilde{f}(x)=\frac{\varepsilon_{0}f(x)}{\varepsilon_{0}\hat{u}(k-1)+\|f\|_{L^{n+1}(Q_{(k-1,k+1)})}},
%~~ x\in Q,
%\end{eqnarray*}
%and satisfies $\|\tilde{f}\|_{L^{n+1}(Q_{(k-1,k+1)})}\leq\varepsilon_{0}$. Now we apply Lemma $\ref{mp02f}$ to $\tilde{u}(x)$ in $Q_{(k-1,k+1)}$, then there exists a constant $\delta\in(0,1)$ such that $\tilde{u}(y,k)\leq1-\delta$ for any $y\in\Omega$, i.e.
%\begin{eqnarray*}
%u(y, k)&\leq&\frac{(1-\delta)}{\varepsilon_{0}}\{\varepsilon_{0}\hat{u}(k-1)
%+\|f\|_{L^{n+1}(Q_{(k-1,k+1)})}\}\\
%&= &(1-\delta)\hat{u}(k-1)+\frac{(1-\delta)}{\varepsilon_{0}}
%\|f\|_{L^{n+1}(Q_{(k-1,k+1)})}\\
%&\leq&(1-\delta)M+
%\frac{(1-\delta)}{\varepsilon_{0}}\|f\|_{L^{n+1}(Q_{(k-1,k+1)})}.
%\end{eqnarray*}
%Hence we have
%$$u^{+}(y, k)\leq(1-\delta)M+
%\frac{(1-\delta)}{\varepsilon_{0}}\|f\|_{L^{n+1}(Q_{(k-1,k+1)})}.
%$$
%By the definition of $\hat{u}(t)$, it follows that
By Corollary $\ref{mpgf}$, there exists $\delta\in(0,1)$ and $0<\varepsilon_{0}<1$ such that for any $k\in\mathbb{Z}$,
\begin{eqnarray*}
\hat{u}(k)&\leq&(1-\delta)\hat{u}(k-1)+
\frac{(1-\delta)}{\varepsilon_{0}}\|f\|_{L^{n+1}(Q_{(k-1,k+1)})}\\ &\leq&(1-\delta)M+
\frac{(1-\delta)}{\varepsilon_{0}}\|f\|_{L^{n+1}(Q_{(k-1,k+1)})}\\ &=&(1-\delta)M+
\frac{(1-\delta)}{\varepsilon_{0}}F.
\end{eqnarray*}
Moreover, for any $k\in\mathbb{Z}$, we have
\begin{eqnarray*}
\hat{u}(k) &\leq&(1-\delta)\hat{u}(k-1)+
\frac{(1-\delta)}{\varepsilon_{0}}F\\
&\leq&(1-\delta)\left((1-\delta)\hat{u}(k-2)+
\frac{(1-\delta)}{\varepsilon_{0}}F\right)+\frac{(1-\delta)}{\varepsilon_{0}}F\\
&=&(1-\delta)^{2}\hat{u}(k-2)+
 \frac{(1-\delta)^{2}}{\varepsilon_{0}}F+\frac{(1-\delta)}{\varepsilon_{0}}F\\
 &\vdots&\\
 &\leq&
(1-\delta)^{m}M+\frac{F}{\varepsilon_{0}}\sum_{i=1}^{m}(1-\delta)^{i}.
\end{eqnarray*}
Then by maximum principle, for any $x=(y,t)\in Q_{[k,k+1)}$ with $k\in\mathbb{Z}$,
\begin{eqnarray*}
\hat{u}(t)&\leq&(1-\delta)^{m}M+\frac{F}{\varepsilon_{0}}\sum_{i=1}^{m}(1-\delta)^{i}+C
\|f\|_{L^{n+1}(Q_{(k-1,k+1)})}\\ &\leq&(1-\delta)^{m}M+\frac{F}{\varepsilon_{0}}\cdot\frac{1-\delta}{\delta}+CF,
\end{eqnarray*}
where $C$ only depends on $n,\lambda,\Lambda,\Omega$. Let $m\rightarrow+\infty$, then for any $t\in\mathbb{R}$, we have
$$\hat{u}(t)\leq \left(\frac {1-\delta}{\varepsilon_{0}\delta}+C\right)F,\quad t\in\mathbb{R}.
$$
Hence for any $x\in Q$,
$$u(x)\leq \left(\frac {1-\delta}{\varepsilon_{0}\delta}+C\right)F,
$$
where $\left(\frac {1-\delta}{\varepsilon_{0}\delta}+C\right)$ depends only on $n,\lambda,\Lambda,\Omega$.
\end{proof}

\begin{lm}[Maximum principle in $Q^{+}$]\label{mpq+f}
Let $u$ satisfy $Lu(x)\leq f(x)$ for $x\in Q^{+}$. If $u(x)$ is bounded from above, then we have
$$\sup_{x\in Q^{+}}u^{+}(x)\leq\sup_{x\in\partial_{p} Q^{+}}u^{+}(x)+C\|f\|_{L^{n+1}_{*}(Q^{+})},
$$
where $C$ depends only on $n,\lambda,\Lambda,\Omega$.
\end{lm}
\begin{proof}
Without loss of generality, we assume that $\sup\limits_{x\in\partial_{p} Q^{+}}u(x)=0$. This implies that $\hat{u}(0)=0$. we denote $\|f\|_{L^{n+1}_{*}(Q^{+})}$ by $F_{1}$. By Lemma $\ref{mpgf}$, there exist $\delta\in(0,1)$ and $\varepsilon_{0}\in(0,1)$ such that for any $k\in\mathbb{N}_{+}$,
\begin{eqnarray*}
\hat{u}(k)
 \leq
(1-\delta)^{k}\hat{u}(0)+\frac{F}{\varepsilon_{0}}\sum_{i=1}^{k}(1-\delta)^{i}.
\end{eqnarray*}
Then by maximum principle, for any $x=(y,t)\in Q_{[k,k+1)}$, $k\in\mathbb{N}_{+}$,
\begin{eqnarray*}
\hat{u}(t)&\leq&(1-\delta)^{k}\hat{u}(0)+\frac{F_{1}}{\varepsilon_{0}}\sum_{i=1}^{k}(1-\delta)^{i}+C
\|f\|_{L^{n+1}(Q_{(k-1,k+1)})}\\
&\leq&(1-\delta)^{k}\hat{u}(0)+\frac{F_{1}}{\varepsilon_{0}}\cdot\frac{1-\delta}{\delta}+CF_{1}\\
&=&\left(\frac{1-\delta}{\varepsilon_{0}\delta}+C\right)F_{1}.
%&\leq&(1-\delta)^{[t]}\hat{u}(0)+\frac{F_{1}}{\varepsilon_{0}}\cdot\frac{1-\delta}{\delta}+CF_{1}.
\end{eqnarray*}

For any $x=(y,t)\in Q_{(0,1)}$, by maximum principle, we have for any $t\in(0,1)$,
$$\hat{u}(t)\leq\hat{u}(0)+CF_{1}=CF_{1}.
$$

The above two inequalities imply the result, hence we complete the proof.
\end{proof}
%\begin{lm}[Maximum principle in $Q^{-}$ with $f$]\label{mpq-f}
%If $u$ satisfies $Lu(x)\leq f(x)$, $\forall x\in Q^{-}$ and $u(x)$ is bounded from above, then we have
%$$\sup_{x\in Q^{-}}u^{+}(x)\leq\sup_{x\in\partial_{l} Q^{-}}u^{+}(x)+C\|f\|_{L^{n+1}_{*}(Q^{-})},
%$$
%where $C$ depends only on $n,\lambda,\Lambda,\Omega$.
%\end{lm}
\begin{lm}[Decay of $u$ in $Q^{+}$ with $f$]\label{decayf}
If $u$ satisfies
\begin{eqnarray*}
\left\{
\begin{array}{rcll}
Lu(x) &\leq& f(x), \qquad &x\in Q^{+},\\
u(x)&=& 0,  \qquad &x\in\partial_{l}Q^{+},
\end{array}
\right.
\end{eqnarray*}
then there exist constants $\alpha,~C_{0},~C_{1}>0$ depending only on $n,\lambda,\Lambda,\Omega$ such that
$$u(x)\leq C_{0}\hat{u}(0)e^{-\alpha t}+C_{1}\|f\|_{L^{n+1}_{*}(Q^{+})},\quad x\in Q^{+}.
$$
\end{lm}
\begin{proof}
We denote $\|f\|_{L^{n+1}_{*}(Q^{+})}$ by $F_{1}$. In fact, in the proof of Lemma $\ref{mpq+f}$, we have already proved that for any $k\in\mathbb{N}_{+}$, $x=(y,t)\in Q_{[k,k+1)}$,
\begin{eqnarray*}
\hat{u}(t)
&\leq&(1-\delta)^{k}\hat{u}(0)+\frac{F_{1}}{\varepsilon_{0}}\cdot\frac{1-\delta}{\delta}+CF_{1}\\
&=&(1-\delta)^{[t]}\hat{u}(0)+\frac{F_{1}}{\varepsilon_{0}}\cdot\frac{1-\delta}{\delta}+CF_{1}.
\end{eqnarray*}
Hence, for any $x=(y,t)\in Q_{[1,+\infty)}$,
\begin{eqnarray*}
\hat{u}(t)&\leq&(1-\delta)^{[t]}\hat{u}(0)+\frac{F_{1}}{\varepsilon_{0}}\cdot\frac{1-\delta}{\delta}+CF_{1}\\
&\leq&(1-\delta)^{t-1}\hat{u}(0)+\frac{F_{1}}{\varepsilon_{0}}\cdot\frac{1-\delta}{\delta}+CF_{1}\\
&=&\frac{\hat{u}(0)}{(1-\delta)}e^{-\alpha t}+\left(\frac{1-\delta}{\varepsilon_{0}\delta}+C\right)F_{1},
\end{eqnarray*}
where $\alpha=-\ln(1-\delta)>0$. For any $x=(y,t)\in Q_{(0,1)}$, by maximum principle, we have for any $t\in(0,1)$,
$$\hat{u}(t)\leq\hat{u}(0)+CF_{1}.
$$
Combining the above two inequalities, we finish the proof by taking $C_{0}=\frac{1}{1-\delta}$ and $C_{1}=\frac{1-\delta}{\varepsilon_{0}\delta}+C$.
\end{proof}
\begin{rem}\label{decayt0f}
For $t_{0}\in\mathbb{R}$, if $u$ satisfies
\begin{eqnarray*}
\left\{
\begin{array}{rcll}
Lu(x) &\leq& f(x), \qquad &x\in Q_{(t_{0},+\infty)},\\
u(x)&=& 0,  \qquad &x\in\partial_{l}Q_{(t_{0},+\infty)},
\end{array}
\right.
\end{eqnarray*}
then there exist constants $\alpha,~C_{0},~C_{1}>0$ depending only on $n,\lambda,\Lambda,\Omega$ such that
$$u(x)\leq C_{0}\hat{u}(t_{0})e^{-\alpha (t-t_{0})}+C_{1}\|f\|_{L^{n+1}_{*}(Q_{(t_{0},+\infty)})},\quad x\in Q_{(t_{0},+\infty)}.
$$
\end{rem}
\section{Harnack inequality}
In this section, in order to study the structure and asymptotic behavior of solutions, we establish a Harnack inequality and a comparison theorem in our form by the boundary Harnack inequality \cite[Theorem 3.5]{Juraj} and the elliptic-type Harnack inequality \cite[Theorem 3.7]{Juraj}.
\begin{lm}\label{bhc}
Let $u\in U$. There exists a constant $C>0$ which depending only on $n,\lambda,\Lambda,\Omega$ such that for any $t_{0}\in\mathbb{R}$,
$$u(x)\leq Cu(0,t_{0}),\quad x\in Q_{(t_{0}-2,t_{0}+2)}.
$$
\end{lm}
\begin{lm}[Comparison Theorem]\label{ct}
Let $u,v\in U$. If $u,v$ satisfy $u(0,1)=v(0,1)$, then there exists constant $C_{*}\geq1$ depending only on $n,\lambda,\Lambda,\Omega$ such that
$$\frac{1}{C_{*}}v(y,t)\leq u(y,t)\leq C_{*}v(y,t),\quad y\in\Omega,~t\geq0.
$$
\end{lm}
\begin{proof}
We denote $A=u(0,1)=v(0,1)$. By elliptic-type Harnack inequality \cite[Theorem 3.7]{Juraj}, there exists $C_{1}>1$ and $r_{0}>0$ small enough depending only on $n,\lambda,\Lambda,\Omega$ such that $(0,1)\in Q^{*}=\{(y,t)|\text{dist}(y,\partial\Omega)> \frac{r_{0}}{2},~t\in(-2,2)\}$ and
$$\frac{A}{C_{1}}\leq u(y,t)\leq C_{1}A,\quad (y,t)\in Q^{*},
$$
$$\frac{A}{C_{1}}\leq v(y,t)\leq C_{1}A,\quad (y,t)\in Q^{*}.
$$
%We denote $Q^{*}=\{(y,t)|\text{dist}(y,\partial\Omega)\geq \frac{r_{0}}{2},~t\in(-2,2)\}$.
For any $(y,t)$ with $\text{dist}(y,\partial\Omega)\leq r_{0},~t=0$, by \cite[Theorem 3.6]{Juraj}, there exists a constant $C_{2}>1$ depending only on $n,\lambda,\Lambda,\Omega$ such that
$$\frac{u(y,0)}{v(y,0)}\leq C_{2}\frac{\sup\limits_{Q^{*}}u}{\inf\limits_{Q^{*}}v}\leq C_{2}\frac{C_{1}A}{\frac{A}{C_{1}}}=C_{2}C_{1}^{2}.
$$
Hence
$$u(y,0)\leq C_{2}C_{1}^{2}v(y,0),\quad y\in\Omega.
$$
Let $C_{*}=C_{2}C_{1}^{2}$ and by maximum principle in $Q^{+}$(Lemma $\ref{mpq+f}$), we obtain that
$$u(y,t)\leq C_{*}v(y,t),\quad (y,t)\in Q^{+}.
$$
By symmetric property, we can also get
$$v(y,t)\leq C_{*}u(y,t),\quad (y,t)\in Q^{+}.
$$
The proof is finished.
\end{proof}
\begin{rem}\label{ctc}
If the condition $u(0,1)=v(0,1)$ is replaced by $u(0,1)\leq v(0,1)$, then we have the following result
$$u(y,t)\leq C_{*}v(y,t),\quad (y,t)\in Q^{+}.
$$
\end{rem}
\begin{rem}\label{kk}
Let $u,v\in U$ with $u(0,t_{0})=v(0,t_{0})$ for some $t_{0}\in\mathbb{R}$. Then there exists constant $C_{*}\geq1$ depending only on $n,\lambda,\Lambda,\Omega$ such that
$$\frac{1}{C_{*}}v(y,t)\leq u(y,t)\leq C_{*}v(y,t),\quad y\in\Omega,~t\geq t_{0}-1.
$$
\end{rem}
The following lemma is an iteration result.
\begin{lm}\label{usc}
Let $u,v\in U$. If $u(0,1)\leq v(0,1)$, then there exists constant $C_{*}\geq1$ depending only on $n,\lambda,\Lambda,\Omega$ such that for any $k\in\mathbb{N}^{+}$,
$$\frac{u(y,t)}{v(y,t)}\leq C_{*}^{k},\quad y\in\Omega,~t\geq 1-k.
$$
\end{lm}
\begin{lm}
Let $u,v\in U$. If $u(0,1)\leq v(0,1)$, then there exists constant $C_{*}\geq1$ depending only on $n,\lambda,\Lambda,\Omega$ such that,
$$\frac{u(y,t)}{v(y,t)}\leq C_{*}^{2},\quad y\in\Omega,~t\leq -1.
$$
\end{lm}
\begin{proof}
We claim that for any $\varepsilon>0$,
$$\frac{u(0,t)}{v(0,t)}\leq C_{*}^{2}+\varepsilon,\quad t\leq0.
$$
In fact, if not, there exist $\varepsilon_{0}>0$ and $t_{0}\in(-\infty,0]$ such that
$$\frac{u(0,t_{0})}{v(0,t_{0})}>C_{*}^{2}+\varepsilon_{0}.
$$
By Remark $\ref{kk}$, it follows that
$$\frac{u(y,t)}{v(y,t)}\geq\frac{C_{*}^{2}+\varepsilon_{0}}{C_{*}}>1,\quad y\in\Omega,~t\geq t_{0}-1,
$$
which contradicts $u(0,1)\leq v(0,1)$.
\end{proof}
The above two lemmas imply the following comparison lemma.
\begin{lm}\label{us}
Let $u,v\in U$. Then there exists constant $C_{*}\geq1$ depending only on $n,\lambda,\Lambda,\Omega$ such that,
$$\frac{1}{C_{*}^{2}}\leq\frac{u(y,t)}{v(y,t)}\frac{v(0,1)}{u(0,1)}\leq C_{*}^{2},\quad (y,t)\in Q.
$$
\end{lm}
%\begin{crl}\label{usc}
%Suppose $u,v\in U$. Then there exists constant $C_{*}\geq1$ depending only on $n,\lambda,\Lambda,\Omega$ such that,
%$$\frac{1}{C_{*}^{4}}\leq\frac{u(y,t)}{v(y,t)}\frac{v(y',t')}{u(y',t')}\leq C_{*}^{4},\quad (y,t),~(y',t')\in Q.
%$$
%\end{crl}
\section{The structure of $U$}
In this section, we show the structure and asymptotic behavior of positive solutions. First, for $u\in U$, we state some properties of $\hat{u}(t)$.
\begin{lm}
For any $u\in U$, $\hat{u}(t)$ is continuous in $(-\infty,+\infty)$.
\end{lm}
\begin{proof}
We only need to prove that $\hat{u}(t)$ is continuous at 0, namely for any $\varepsilon>0$, there exists $\delta>0$, such that for any $0<|t|<\delta$,
$$|\hat{u}(t)-\hat{u}(0)|<\varepsilon.
$$
Since $u\in U$, we have $\hat{u}(0)>0$. Then there exists $y_{0}\in\Omega$ such that $u(y_{0},0)=\sup\limits_{y\in\Omega}u(y,0)=\hat{u}(0)$.

On the one hand, since $u(x)$ is continuous in $\bar{Q}$, for any $\varepsilon>0$, there exists $\delta_{1}>0$, such that for any $x$ satisfying $\text{dist}_{p}(x,(y_{0},0))<\sqrt{\delta_{1}}$,
$$u(x)>u(y_{0},0)-\varepsilon.
$$
Especially, for $x=(y_{0},t)$ with $|t|<\delta_{1}$, we have $u(y_{0},t)>u(y_{0},0)-\varepsilon$. By the definition of $\hat{u}$, we have
\begin{equation}\label{kl}
\hat{u}(t)>\hat{u}(0)-\varepsilon,\quad |t|<\delta_{1}.	
\end{equation}

On the other hand, $u(y_{0},0)=\hat{u}(0)$ implies that for any $y\in\Omega$,
\begin{equation}
\label{yy0}u(y,0)\leq u(y_{0},0).
\end{equation}
Note that for any $y\in\Omega$, $u$ is continuous at $(y,0)$, then for the above $\varepsilon>0$, there exists $\delta_{y}>0$, such that for any $x$ satisfying $\text{dist}_{p}(x,(y,0))<\sqrt{\delta_{y}}$,
\begin{equation}\label{xy0}
u(x)-u(y,0)<\frac{\varepsilon}{2}.
\end{equation}
Since $\{(y,0)|y\in\Omega\}$ is a compact set in $\mathbb{R}^{n+1}$, it can be covered by $\bigcup\limits_{y\in\Omega}O_{p}((y,0),\frac{\sqrt{\delta_{y}}}{2})$. By the finite covering theorem, the set $\{(y,0)|y\in\Omega\}$ is covered by $\bigcup\limits_{i=1}^{m}O_{p}((y_{i},0),\frac{\sqrt{\delta_{y_{i}}}}{2})$. Let $\delta_{2}=\frac{1}{2}\min\limits_{1\leq i\leq m}\sqrt{\delta_{y_{i}}}$, for any $x=(y,t)$ with $|t|<\delta_{2}$, $y\in\Omega$, there exists $1\leq i\leq m$ such that $(y,t)\in O_{p}((y_{i},0),\sqrt{\delta_{y_{i}}})$.
Then it follows from $(\ref{yy0})$ and $(\ref{xy0})$ that
$$u(y,t)<u(y_{i},0)+\frac{\varepsilon}{2}\leq u(y_{0},0)+\frac{\varepsilon}{2}.
$$
This means that for any $|t|<\delta_{2}$,
\begin{equation}\label{kz}
\hat{u}(t)\leq u(y_{0},0)+\frac{\varepsilon}{2}<\hat{u}(0)+\varepsilon.	
\end{equation}

Let $\delta=\min\{\delta_{1},\delta_{2}\}$. By (\ref{kl}) and (\ref{kz}), we get that for any $|t|<\delta$,
$$|\hat{u}(t)-\hat{u}(0)|<\varepsilon.
$$
\end{proof}
\begin{lm}
For any $u\in U$, we have $m(u)=0$, $\hat{u}(t)$ is a strictly decreasing function in $\mathbb{R}$.
\end{lm}
\begin{proof}
We first prove that $\hat{u}(t)$ is a strictly decreasing function in $\mathbb{R}$. In fact, for any $-\infty<t_{1}<t_{2}<+\infty$, by using maximum principle for $u$ in $Q_{(t_{1},t_{2})}$, we have $0<\hat{u}(t_{2})\leq\hat{u}(t_{1})$. If $\hat{u}(t_{2})=\hat{u}(t_{1})>0$, then by strong maximum principle \cite[Theorem 2.7]{Li}, $u(x)=\hat{u}(t_{1})>0$ for any $x\in Q_{(t_{1},t_{2})}$. By continuity of $u$, it follows that $u(x)=\hat{u}(t_{1})>0$ on $\partial_{l}Q_{(t_{1},t_{2})}$, which contradicts $u=0$ on $\partial_{l}Q_{(t_{1},t_{2})}$. Hence $\hat{u}(t)$ is strictly decreasing in $\mathbb{R}$.

Now we prove $m(u)=0$. %By the definitions of $m(u)$ and $\hat{u}(t)$, there exists a minimizing sequence $\{x_{j}=(y_{j},t_{j})\}_{j=1}^{\infty}\subset Q$ such that
%$$u(x_{j})=u(y_{j},t_{j})=\hat{u}(t_{j})=\sup_{y\in\Omega}u(y,t_{j}),\quad \lim_{j\rightarrow\infty}u(x_{j})=\lim_{j\rightarrow\infty}\hat{u}(t_{j})=m(u)\geq0.
%$$
%Since $\hat{u}(t)$ is continuous in $\mathbb{R}$, then the limits of all the subsequences of $\{\hat{u}(t_{j})\}_{j=1}^{\infty}$ equal to $m(u)$.
By Lemma \ref{decayf}, we get that $\lim\limits_{t\rightarrow+\infty}\hat{u}(t)=0$. This implies that $m(u)=0.$ %we claim that there exists a subsequence of $\{x_{j}=(y_{j},t_{j})\}_{j=1}^{\infty}$, still denoted by $\{x_{j}=(y_{j},t_{j})\}_{j=1}^{\infty}$, and $x^{*}=(y^{*},t^{*})\in Q$ such that
%$$\lim_{j\rightarrow\infty}x_{j}=x^{*},\quad \lim_{j\rightarrow\infty}u(x_{j})=u(x^{*})=m(u)>0.
%$$

%Now we prove the claim by contradiction. On the one hand, if there exists a subsequence of $\{x_{j}=(y_{j},t_{j})\}_{j=1}^{\infty}$, still denoted by $\{x_{j}=(y_{j},t_{j})\}_{j=1}^{\infty}$ such that
%\begin{equation}\label{t+infty}
%\lim_{j\rightarrow\infty}t_{j}=+\infty,\quad \lim_{j\rightarrow\infty}\hat{u}(t_{j})=m(u)>0.
%\end{equation}
%Then by the exponential decay of $u$ in $Q^{+}$(Lemma $\ref{decayf}$), we have
%$$\lim_{t\rightarrow+\infty}\hat{u}(t)=0,
%$$
%which contradicts $(\ref{t+infty})$ by the continuity of $\hat{u}(t)$. On the other hand, if there exists a subsequence of $\{x_{j}=(y_{j},t_{j})\}_{j=1}^{\infty}$, still denoted by $\{x_{j}=(y_{j},t_{j})\}_{j=1}^{\infty}$ such that
%$$
%\lim_{j\rightarrow\infty}t_{j}=-\infty,\quad \lim_{j\rightarrow\infty}\hat{u}(t_{j})=m(u)>0.
%$$
%This means that for any $t\in\mathbb{R}$, $\hat{u}(t)=m(u)$, and hence $u\leq m(u)$ in $Q$. Then by Lemma $\ref{mpqf}$, $u\equiv0$ in $Q$, which contradicts $u>0$ in $Q$. Thus we prove the claim.
\end{proof}

%Next by the boundary H\"{o}lder estimate, there exist $C'_{H}>0$ and $0<\alpha<1$, which are only dependent of $n,\lambda,\Lambda,\Omega,H$, such that for any $R>H+1$,
%$$[u_{R}]_{C^{\alpha}(Q_{(-H,H)})}\leq C'_{H}.
%$$

%For $H=1$, by the Arzel\'{a}-Ascoli Theorem, there exists a subsequence of $\{u_{R}\}_{R=1}^{\infty}$, denoted by $\{u^{(1)}_{R}\}_{R=1}^{\infty}$, such that $\{u^{(1)}_{R}\}$ converges uniformly in $Q_{(-1,1)}$. For $H=2$, there also exists a subsequence of $\{u^{(1)}_{R}\}_{R=1}^{\infty}$, denoted by $\{u^{(2)}_{R}\}_{R=1}^{\infty}$ such that $\{u^{(2)}_{R}\}$ converges uniformly in $Q_{(-2,2)}$. To continue, for any $H\in\mathbb{N}_{+}$, there exists a sequence $\{u^{(H)}_{R}\}_{R=1}^{\infty}$ converges uniformly in $Q_{(-H,H)}$. Then the dialogue sequence $\{u^{(R)}_{R}\}_{R=1}^{\infty}$ converges uniformly in $Q_{(-H,H)}$ for any $H\in\mathbb{N}_{+}$, thus there exists a function $u(x)$ such that $u_{R}$ uniformly converges to $u(x)$ in $W^{2,1}_{n+1,\text{loc}}(Q)\cap C(\bar{Q})$. It follows that $u$ is a solution of
%\begin{eqnarray*}
%\left\{
%\begin{array}{rcll}
%Lu(x) &=& 0, \qquad &x\in Q,\\
%u(x)&=& 0,  \qquad &x\in\partial Q.\\
%\end{array}
%\right.
%\end{eqnarray*}
%And we also have $u(0,0)=1$ since $u_{R}(0,0)=1$ for any $R\in\mathbb{N}^{+}$. Then by the elliptic-type Harnack inequality (Theorem 3.7 in \cite{Juraj}), we have for any $Q'\subset Q$ with $(0,0)\in Q'$, $u(x)>0$ in $Q'$. It follows that $u>0$ in $Q$. This implies that $u\in U$.
%\end{proof}

Now we establish the structure of the positive solution set $U$.

\
\

{\bf The proof of Theorem \ref{se0}.} By Lemma \ref{sef}(see Section 5), one gets that $U$ is well defined. For any $u,v\in U$, we set
$$E=\{k>0|u(x)\leq kv(x), x\in Q\},\quad K=\inf E.
$$
By Lemma $\ref{us}$, we know $\frac{u(0,1)}{v(0,1)}C_{*}^{2}\in E$, so $E\neq\emptyset$ and $K\geq 0$. Note that $Kv(x)-u(x)\geq0$ for any $x\in Q$, if $K=0$, then $u(x)\leq0$. This is a contradiction to that $u>0$ in $Q$, hence $K>0$.

Now we claim that
$$Kv(x)-u(x)=0,\quad x\in Q.
$$
We prove the claim by contradiction. If $Kv(x)-u(x)>0$, this implies that $Kv-u\in U$. Then by Lemma $\ref{us}$, there exists a constant $K_{1}\geq1$ such that $v(x)\leq K_{1}(Kv(x)-u(x))$ in $Q$, i.e.
$$(K-\frac{1}{K_{1}})v(x)-u(x)\geq0.
$$
This means that $K-\frac{1}{K_{1}}\in E$, which contradicts the definition of $K$. Hence we get that $u=Kv$ in $Q$.\qed

\
\

In the following, we prove the asymptotic behavior of positive solutions.

\
\

{\bf The proof of Theorem \ref{th2}.} Firstly, we claim that for any $u\in U$, there exists a constant $\theta>0$ depending only on $n,\lambda,\Lambda,\Omega$ such that
%\begin{equation}\label{i1}
%\hat{u}(t+1)\leq(1+\theta)\hat{u}(t),~~ t\in(-\infty,+\infty),
%\end{equation}
\begin{equation}\label{i2}
\hat{u}(t-1)\leq(1+\theta)\hat{u}(t),~~t\in(-\infty,+\infty),
\end{equation}
In fact, by Lemma $\ref{bhc}$, there exists a constant $C>0$ depending only on $n,\lambda,\Lambda,\Omega$ such that for any $t_{0}\in\mathbb{R}$,
$$u(x)\leq Cu(0,t_{0})\leq C\hat{u}(t_{0}),\quad x\in Q_{(t_{0}-2,t_{0}+2)}.
$$
Hence there exists $\theta>0$ depending only on $n,\lambda,\Lambda,\Omega$ such that
$$\hat{u}(t_{0}-1)\leq(1+\theta)\hat{u}(t_{0}), t_{0}\in\mathbb{R}.
$$
%$$\hat{u}(t_{0}+1)\leq(1+\theta)\hat{u}(t_{0}).$$
Next we claim that for any $u\in U$, there exists a constant $\eta$ depending only on $n,\lambda,\Lambda,\Omega$ such that
\begin{equation}\label{i3}
(1+\eta)\hat{u}(t+1)\leq\hat{u}(t),~~ t\in(-\infty,+\infty).
\end{equation}
We prove this claim by contradiction. If not, then for any $k\in\mathbb{N}^{+}$, there exists $t_{k}\in\mathbb{R}$ such that
$$(1+\frac{1}{k})\hat{u}(t_{k}+1)>\hat{u}(t_{k}).
$$
For $u$ in $Q_{(t_{k},t_{k}+2)}$, by Corollary $\ref{mpgf}$, we get that
$$\hat{u}(t_{k}+1)\leq(1-\delta)\hat{u}(t_{k})<(1-\delta)(1+\frac{1}{k})\hat{u}(t_{k}+1).
$$
Let $k$ large enough such that $(1-\delta)(1+\frac{1}{k})<1$, then
$$\hat{u}(t_{k}+1)<(1-\delta)(1+\frac{1}{k})\hat{u}(t_{k}+1)<\hat{u}(t_{k}+1).$$
This is a contradiction. Hence (\ref{i3}) holds.

By $(\ref{i2})$, we have
$$\hat{u}(t)\geq (1+\theta)^{-([t]+1)}\hat{u}(t-[t]-1)\geq (1+\theta)^{-(t+1)}\hat{u}(0),\quad t\in(0,+\infty).
$$
$$\hat{u}(t)\leq (1+\theta)^{[|t|]+1}\hat{u}(t+[|t|]+1)\leq (1+\theta)^{|t|+1}\hat{u}(0),\quad t\in(-\infty,0).
$$

By $(\ref{i3})$, we have
$$\hat{u}(t)\leq (1+\eta)^{-[t]}\hat{u}(t-[t])\leq (1+\eta)^{-(t-1)}\hat{u}(0)=(1+\eta)^{-t+1}\hat{u}(0),\quad t\in(0,+\infty)
$$
$$\hat{u}(t)\geq (1+\eta)^{[|t|]}\hat{u}(t+[|t|])\geq (1+\eta)^{|t|-1}\hat{u}(0),\quad t\in(-\infty,0).
$$

By taking $\alpha=\ln(1+\theta)$, $\beta=\ln(1+\eta)$, $C=\frac{1}{1+\eta}$, $C'=1+\theta$, then we get that $(\ref{e1})$-$(\ref{e2})$ hold.\qed

\
\

{\bf The proof of Theorem \ref{th3}.} For any $j\in\mathbb{N}_{+}$, we define
$$E_{j}=\{k>0|u(x)\leq kw(x),~x\in Q_{(j,+\infty)}\},\quad K_{j}=\inf E_{j},
$$
$$F_{j}=\{l>0|u(x)\geq lw(x),~x\in Q_{(j,+\infty)}\},\quad L_{j}=\sup F_{j}.
$$
By Lemma $\ref{us}$, we have that $\frac{u(0,1)}{w(0,1)}C_{*}^{2}\in E_{j}$ and $\frac{u(0,1)}{w(0,1)C_{*}^{2}}\in F_{j}$ for any $j\in\mathbb{N}_{+}$. This means that $E_{j}\neq\emptyset$ with a lower bound 0 and $F_{j}\neq\emptyset$ with a upper bound $\frac{u(0,1)}{w(0,1)}C_{*}^{2}$. It follows that $0< L_{j}\leq K_{j}<+\infty$.

Now we claim that there exists a constant $0<\zeta<1$ depending only on $n,\lambda,\Lambda,\Omega$ such that
\begin{equation}\label{kj}
K_{j+1}-L_{j+1}\leq\zeta(K_{j}-L_{j}).
\end{equation}
In fact, for any $y\in\Omega$ and $j\in\mathbb{N}_{+}$,
$$0<L_{j}w(y,j+1)\leq L_{j+1}w(y,j+1)\leq u(y,j+1)\leq K_{j+1}w(y,j+1)\leq K_{j}w(y,j+1).
$$
Then it follows that
\begin{equation}\label{uw1}
u(y,j+1)\geq L_{j}w(y,j+1)+\frac{1}{2}(K_{j}-L_{j})w(y,j+1),
\end{equation}
or
\begin{equation}\label{uw2}
u(y,j+1)\leq K_{j}w(y,j+1)-\frac{1}{2}(K_{j}-L_{j})w(y,j+1).
\end{equation}
If $u$ satisfies $(\ref{uw1})$, by Lemma $\ref{usc}$, there exists constant $C_{*}\geq1$ such that
$$u(x)-L_{j}w(x)\geq \frac{1}{2C_{*}}(K_{j}-L_{j})w(x),\quad x=(y,t)\in Q_{[j+1,+\infty)}.
$$
Then we get that
$$L_{j}w(x)+\frac{1}{2C_{*}}(K_{j}-L_{j})w(x)\leq u(x),\quad x=(y,t)\in Q_{[j+1,+\infty)}.
$$
By the definition of $L_{j+1}$, we get that
$$L_{j+1}\geq L_{j}+\frac{1}{2C_{*}}(K_{j}-L_{j}).$$
Thus we obtain that
\begin{equation}\label{uw3}
K_{j+1}-L_{j+1}\leq K_{j}-\left(L_{j}+\frac{1}{2C_{*}}(K_{j}-L_{j})\right)=(1-\frac{1}{2C_{*}})(K_{j}-L_{j}).
\end{equation}
If $(\ref{uw2})$ is satisfied, by similar arguments, we can also get that $(\ref{uw3})$ holds.

Note that $\{K_{j}\}_{j=1}^{\infty}$ is a decreasing sequence and $\{L_{j}\}_{j=1}^{\infty}$ is an increasing sequence. By (\ref{kj}), we get that
$$K_{j}-L_{j}\leq \zeta^{j-1}(K_{1}-L_{1})\leq Ce^{j\ln\zeta}.
$$
Hence there exists a constant $K>0$ such that $\lim\limits_{j\rightarrow+\infty}K_{j}=K=\lim\limits_{j\rightarrow+\infty}L_{j}$.
Thus we can calculate that for any $x=(y,t)\in Q_{(1,+\infty)}$,
$$|u(x)-Kw(x)|\leq(K_{[t]}-L_{[t]})w(x)\leq Ce^{[t]\ln\zeta}w(x)\leq \frac{C}{\zeta}e^{-\alpha t}w(x),\quad \alpha=-\ln\zeta.
$$
\qed

\section{The structure of $\tilde{U}$}
In this section, we establish the structure of $\tilde{U}$. First, we demonstrate a result about the existence and uniqueness of bounded solutions in $Q$.
\begin{lm}\label{sef}
Let $f\in L_{\text{loc}}^{n+1}(Q)$ with $\|f\|_{L_{*}^{n+1}(Q)}<+\infty$. Then the following Dirichlet problem
\begin{equation}\label{sol}
\left\{
\begin{array}{rcll}
Lu(x) &=& f(x), \qquad &x\in Q,\\
u(x)&=& 0,  \qquad &x\in \partial Q, \\
%u(y,t)&& \hspace{-8mm}\mbox{is uniformly bounded from below}, \qquad &(y,t)\in Q,
\end{array}
\right.
\end{equation}
has a unique bounded solution $u\in W^{2,1}_{n+1,\text{loc}}(Q)\cap C(\bar{Q})$.
\end{lm}
\begin{proof}
For any $N\in\mathbb{N}_{+}$, consider the following Dirichlet problem in $Q_{(-N,N)}$:
\begin{eqnarray*}
\left\{
\begin{array}{rcll}
Lu(x) &=& f(x), \qquad &x\in Q_{(-N,N)},\\
u(x)&=& 0,  \qquad &x\in\partial_{p} Q_{(-N,N)}.\\
\end{array}
\right.
\end{eqnarray*}
By the classical existence theory \cite[Theorem 7.32]{Li}, we get that there exists a unique solution $u_{N}\in W_{n+1,\text{loc}}^{2,1}(Q_{(-N,N)})\cap C(\overline{Q_{(-N,N)}})$. By maximum principle, we have
$$\|u_{N}\|_{L^{\infty}(Q_{(-N,N)})}\leq C_{N}\|f\|_{L^{n+1}(Q_{(-N,N)})},
$$
where $C_{N}$ depends only on $n,\lambda,\Lambda,\Omega,N$.

In the following, we prove that there exists a constant $C_{0}>0$ not depending on $N$ such that
$$\|u_{N}\|_{L^{\infty}(Q_{(-N,N)})}\leq C_{0}\|f\|_{L_{*}^{n+1}(Q)}.
$$
For convenience, we denote $M=\|u_{N}\|_{L^{\infty}(Q_{(-N,N)})}$. For any $\xi\in[-N+1,N-1]$, $Q_{(\xi-1,\xi+1)}\subset Q_{(-N,N)}$. By using Corollary $\ref{mpgf}$ to $u_{N}$ with $t_{0}=\xi-1$, we have
\begin{eqnarray*}
u_{N}(y,\xi)\leq(1-\delta)M+\frac{1-\delta}{\varepsilon_{0}}\|f\|_{L^{n+1}(Q_{(\xi-1,\xi+1)})},~~ y\in\Omega.
\end{eqnarray*}
Take the supreme of $\xi$ in $(-N+1,N-1)$, it follows that
\begin{eqnarray*}
\sup\limits_{\xi\in(-N+1,N-1)}\hat{u}_{N}(\xi)&\leq&(1-\delta)M+\frac{1-\delta}{\varepsilon_{0}}
\sup\limits_{\xi\in(-N+1,N-1)}\|f\|_{L^{n+1}(Q_{(\xi-1,\xi+1)})}\\
&\leq&(1-\delta)M+\frac{1-\delta}{\varepsilon_{0}}\|f\|_{L_{*}^{n+1}(Q)}.
\end{eqnarray*}
That is
\begin{eqnarray*}
\sup\limits_{x\in Q_{(-N+1,N-1)}}u_{N}(x)\leq(1-\delta)M+\frac{1-\delta}{\varepsilon_{0}}\|f\|_{L_{*}^{n+1}(Q)}.
\end{eqnarray*}
Furthermore, we have
$$\sup\limits_{x\in Q_{(-N+1,N-1)}}|u_{N}(x)|\leq(1-\delta)M+\frac{1-\delta}{\varepsilon_{0}}\|f\|_{L_{*}^{n+1}(Q)}.
$$
For any $x\in Q_{(-N,-N+1)}$, by maximum principle, we have
\begin{eqnarray*}
\|u_{N}\|_{L^{\infty}(Q_{(-N,-N+1)})}&\leq&\|u_{N}\|_{L^{\infty}(\partial_{p}Q_{(-N,-N+1)})}+ C_{1}\|f\|_{L^{n+1}(Q_{(-N,-N+1)})}\\
&\leq&C_{1}\|f\|_{L^{n+1}(Q_{(-N,-N+1)})}\\
&\leq&C_{1}\|f\|_{L_{*}^{n+1}(Q)}.
\end{eqnarray*}
Similarly, for any $x\in Q_{(N-1,N)}$, by maximum principle, we have
\begin{eqnarray*}
\|u_{N}\|_{L^{\infty}(Q_{(N-1,N)})}&\leq&\|u_{N}\|_{L^{\infty}(\partial_{p}Q_{(N-1,N)})}+ C'_{1}\|f\|_{L^{n+1}(Q_{(N-1,N)})}\\
&\leq&\sup\limits_{x\in Q_{(-N+1,N-1)}}|u_{N}(x)|+C'_{1}\|f\|_{L^{n+1}(Q_{(N-1,N)})}\\
&\leq&(1-\delta)M+\frac{1-\delta}{\varepsilon_{0}}\|f\|_{L_{*}^{n+1}(Q)}+ C'_{1}\|f\|_{L_{*}^{n+1}(Q)}
\\&=&
(1-\delta)M+(\frac{1-\delta}{\varepsilon_{0}}+C'_{1})\|f\|_{L_{*}^{n+1}(Q)}.
\end{eqnarray*}
By taking $C_{2}=\max\{C_{1},\frac{1-\delta}{\varepsilon_{0}}+C'_{1}\}$, it follows that
$$\|u_{N}\|_{L^{\infty}(Q_{(-N,N)})}\leq (1-\delta)M+C_{2}\|f\|_{L_{*}^{n+1}(Q)},
$$
where $C_{2}$ depends only on $n,\lambda,\Lambda,\Omega$. The above inequality implies that $M\leq (1-\delta)M+C_{2}\|f\|_{L_{*}^{n+1}(Q)}$, i.e.
$$\|u_{N}\|_{L^{\infty}(Q_{(-N,N)})}=M\leq C_{0}|f\|_{L_{*}^{n+1}(Q)},
$$
where $C_{0}=\frac{C_{2}}{\delta}$ depends only on $n,\lambda,\Lambda,\Omega$.

Hence for any $L\in\mathbb{N}_{+}$ with $L<N$, we have
$$\|u_{N}\|_{L^{\infty}(Q_{(-L,L)})}\leq\|u_{N}\|_{L^{\infty}(Q_{(-N,N)})}\leq C_{0}|f\|_{L_{*}^{n+1}(Q)}.
$$
By the boundary H\"{o}lder estimate, there exists a constant $C_{*}>0$ depending only on $n,\lambda,\Lambda,\Omega,L$ and $0<\alpha<1$ such that
$$[u_{N}]_{C^{\alpha}(\overline{Q_{(-L,L)}})}\leq C_{*}.
$$
For $L=1$, by the Arzel\'{a}-Ascoli Theorem, there exists a subsequence of $\{u_{N}\}_{N=1}^{\infty}$, denoted by $\{u^{(1)}_{N}\}_{N=1}^{\infty}$, such that $\{u^{(1)}_{N}\}$ converges uniformly in $Q_{(-1,1)}$. For $L=2$, there also exists a subsequence of $\{u^{(1)}_{N}\}_{N=1}^{\infty}$, denoted by $\{u^{(2)}_{N}\}_{N=1}^{\infty}$ such that $\{u^{(2)}_{N}\}$ converges uniformly in $Q_{(-2,2)}$. To continue, for any $L\in\mathbb{N}_{+}$, there exists a sequence $\{u^{(L)}_{N}\}_{N=1}^{\infty}$ converges uniformly in $Q_{(-L,L)}$. Then the dialogue sequence $\{u^{(N)}_{N}\}_{N=1}^{\infty}$ converges uniformly in $Q_{(-L,L)}$ for any $L\in\mathbb{N}_{+}$, thus there exists a function $u(x)$ such that $u_{N}$ uniformly converges to $u(x)$ in $W^{2,1}_{n+1,\text{loc}}(Q)\cap C(\bar{Q})$.
Therefore, $u$ is bounded solution of $(\ref{sol})$. By Lemma \ref{mpqf}, we get the uniqueness of the solution.
\end{proof}

\
\

{\bf The proof of Theorem \ref{th4}.} By Lemma $\ref{sef}$, there exists a unique bounded solution $v\in W^{2,1}_{n+1,\text{loc}}(Q)\cap C(\overline{Q})$ for the following problem:
\begin{equation*}
\left\{
\begin{array}{rcll}
Lv(x) &=& f(x), \qquad &x\in Q,\\
v(x)&=& 0,  \qquad &x\in \partial Q. \\
%u(y,t)&& \hspace{-8mm}\mbox{is uniformly bounded from below}, \qquad &(y,t)\in Q,
\end{array}
\right.
\end{equation*}
Since $u$ is bounded from below, there exists a constant $C>0$ such that $u-v\geq -C$ in $Q$. Moreover, one gets that $u-v$ satisfies
\begin{eqnarray*}
\left\{
\begin{array}{rcll}
L(u-v)(x) &=& 0, \qquad &x\in Q, \\
(u-v)(x)&=& 0, \qquad &x\in\partial Q,\\
(u-v)(x)&\geq& -C, \qquad &x\in Q.
\end{array}
\right.
\end{eqnarray*}
Since $v-u$ is bounded from above, by Lemma $\ref{mpqf}$, we get that $u-v\geq0$. Thus either $u\equiv v$ or $u-v>0$. If $u=v$, then our conclusion clearly holds by taking $a=0$. If $u-v>0$, by Theorem $\ref{se0}$, there exists $w\in U$ such that $u-v=aw$, that is $u=v+aw$. Therefore we obtain our result:
$$\tilde{U}=U^{0}+U=\{u_{0}+au|a\geq0\},
$$
where $U^{0}=\{v\}$ is the unique bounded solution to $Lu_{0}=f$ in $Q$ with zero boundary condition.\qed

%\section{Coefficients and Nonhomogenous term: Maximum principle, Harnack inequality}
%\begin{lm}[Growth of $u$ in $Q^{-}$ with $f$]\label{decayf}
%Suppose $u$ satisfies
%\begin{eqnarray*}
%\left\{
%\begin{array}{rcll}
%Lu(x) &\leq& f(x), \qquad &x\in Q^{-},\\
%u(x)&=& 0,  \qquad &x\in\partial_{l}Q^{-},
%\end{array}
%\right.
%\end{eqnarray*}
%then there exist constants $\alpha,~C_{0},~C_{1}>0$ depending only on $n,\lambda,\Lambda,\Omega$ such that
%$$u(x)\geq C_{0}\hat{u}(0)e^{\alpha |t|}-C_{1}\|f\|_{L^{n+1}_{*}(Q^{-})},\quad x\in Q^{-}.
%$$
%\end{lm}
%\begin{proof}
%In the following we denote $\|f\|_{L^{n+1}_{*}(Q^{-})}$ by $F_{2}$. In fact, in the proof of Lemma $\ref{mpqf}$, we have already proved that
%Then for any $k\in\mathbb{N}_{-}$,
%\begin{eqnarray*}
%\hat{u}(0) \leq
%(1-\delta)^{|k|}\hat{u}(k)+\frac{F}{\varepsilon_{0}}\sum_{i=1}^{|k|}(1-\delta)^{i}.
%\end{eqnarray*}
%And by maximum principle, for any $x=(y,t)\in Q_{[k-1,k)}$, $k\in\mathbb{N}_{-}$,
%\begin{eqnarray*}
%(1-\delta)^{-|k|}\hat{u}(0)-\frac{F}{\varepsilon_{0}}\sum_{i=1}^{|k|}(1-\delta)^{i-|k|}
%&\leq&\hat{u}(k)\\
%&\leq&\hat{u}(t)+C\|f\|_{L^{n+1}(Q_{(k-1,k+1)})}\\
%&\leq&\hat{u}(t)+CF.
%\end{eqnarray*}
%Hence, for any $x=(y,t)\in Q^{+}$,
%\begin{eqnarray*}
%\hat{u}(t)&\geq&(1-\delta)^{-[|t|]}\hat{u}(0)-\frac{F}{\varepsilon_{0}}\cdot\frac{1-\delta}{\delta}-CF\\
%&\geq&(1-\delta)^{-|t|+1}\hat{u}(0)-\frac{F}{\varepsilon_{0}}\cdot\frac{1-\delta}{\delta}-CF\\
%&=&\hat{u}(0)(1-\delta)e^{\alpha |t|}-(\frac{1-\delta}{\varepsilon_{0}\delta}+C)F,
%\end{eqnarray*}
%where $\alpha=-\ln(1-\delta)$. Then we finish the proof.
%\end{proof}

\section{Acknowledgements}
The authors thank Professors Lihe Wang and Chunqin Zhou for helpful discussions and suggestions. The authors also would like to thank the anonymous reviewers' time and comments on this paper.

\
 \

 {\bf Data availability} Not applicable.

 \
 \

{\bf Declarations}

\

{\bf Conflict of interest} The authors state that there are no conflicts of interests in the preparation of manuscript.


\begin{thebibliography}{99}

\bibitem{BWZ} J. Bao, L. Wang and C. Zhou, Positive solutions to elliptic equations in unbounded cylinder. Discrete Contin. Dyn. Syst. Ser. B 21 (2016), no. 5, 1389-1400.


\bibitem{B} M. Benedicks, Positive harmonic functions vanishing on the boundary of certain domains in $\mathbb{R}^n$. Ark. Mat. 18 (1980), no. 1, 53-72.


\bibitem{CC} G. Ciraolo, R. Corso and A. Roncoroni, Classification and non-existence results for weak solutions to quasilinear elliptic equations with Neumann or Robin boundary conditions. J. Funct. Anal. 280 (2021), no. 1, Paper No. 108787, 27 pp.

\bibitem{F} G. Feng, Counting ancient solutions on a strip with exponential growth.  Math. Res. Lett. 29 (2022), 1445-1459.


\bibitem{GP} M. Ghergu, J. Pres, Positive harmonic functions that vanish on a subset of a cylindrical surface. Potential Anal. 31 (2009), no. 2, 147-181.


\bibitem{Gr}
M. Gruber, Harnack inequalities for solutions of general second order parabolic equations and estimates of their H\"{o}lder constants, Math. Z. 185 (1984), 23-43.

\bibitem{HL} F. Han, L. Wang, Positive solutions to discrete harmonic functions in unbounded cylinders. J. Korean Math. Soc. 61 (2024), no. 2, 377-393.


\bibitem{HL1} F. Hang, F. Lin, Exponential growth solutions of elliptic equations. Acta Math. Sin. (Engl.
Ser.), 15 (1999), 525-534.


\bibitem{H3} X. T. Huang, Counting dimensions of $L$-harmonic functions with exponential growth, Geometriae Dedicata, 209 (2020), 31-42.


\bibitem{Juraj}
J. H\'{u}ska, P. Pol\'{a}\v{c}ik, M. V. Safonov, Harnack inequalities, exponential separation, and perturbations of principal Floquet bundles for linear parabolic equations, Ann. Inst. H. Poincar\'{e} Anal. Non Lin\'{e}aire 24 (2007), 711-739.

\bibitem{LN}  E. M. Landis, N. S. Nadirashvili, Positive solutions of second-order equations in unbounded domains. Mat. Sb. 126(168) (1985), no. 1, 133-139, 144.



\bibitem{Li}
G. M. Lieberman, Second Order Parabolic Differential Equations, World Scientific Publishing Co. Inc., River Edge, NJ, 1996.



\bibitem{M} M. Murata,
On construction of Martin boundaries for second order elliptic equations.
Publ. Res. Inst. Math. Sci. 26 (1990), no. 4, 585-627.


\bibitem{LWZ} L. Wang, L. Wang and C. Zhou, The exponential growth and decay properties for solutions to elliptic equations in unbounded cylinders. J. Korean Math. Soc. 57 (2020), no. 6, 1573-1590.

\bibitem{LWZ1}  L. Wang, L. Wang and C. Zhou, Classification of positive solutions for fully nonlinear elliptic equations in unbounded cylinders. Commun. Pure Appl. Anal. 20 (2021), no. 3, 1241-1261.

\bibitem{L1} L. Wang, The exponential property of solutions bounded from below to degenerate equations in unbounded domains. Acta Math. Sci. Ser. B (Engl. Ed.) 42 (2022), no. 1, 323-348.

\bibitem{LWZ2} L. Wang, L. Wang, C. Zhou and Z. Li, The behavior and classification of solutions bounded from below to degenerate elliptic equations in unbounded cylinders. J. Math. Anal. Appl. 516 (2022), no. 2, Paper No. 126560, 31 pp.


\bibitem{LWZ3} L. Wang, L. Wang and C. Zhou, The dimensional estimates of exponential growth solutions to uniformly elliptic equations of non-divergence form. Discrete Contin. Dyn. Syst. 42 (2022), no. 11, 5223-5238.






\end{thebibliography}
\end{document}